\documentclass[12pt]{article}

\usepackage[usenames]{color} 

\usepackage{amssymb} 
\usepackage{amsmath} 
\usepackage{amssymb,bm,amsmath,amssymb,amsthm}
\usepackage[]{inputenc}
\usepackage[small,nohug,heads=vee]{diagrams}
\diagramstyle[labelstyle=\scriptstyle]

\usepackage{authblk}

\usepackage[all]{xy}
  \SelectTips{cm}{10}     
  \everyxy={<2.5em,0em>:} 

\usepackage{hyperref}

\DeclareMathOperator{\Gal}{Gal}

\DeclareMathOperator{\CL}{Cl}
\DeclareMathOperator{\tor}{tors}

\DeclareMathOperator{\Hom}{Hom}

\DeclareMathOperator{\T}{\mathbb T}

\newtheorem{cor}{Corollary}

\newtheorem{lem}{Lemma}

\newtheorem{theorem}{Theorem}
\newcommand{\overbar}[1]{\mkern 1.5mu\overline{\mkern-1.5mu#1\mkern-1.5mu}\mkern 1.5mu}

\begin{document}
\title{A remark On Abelianized Absolute Galois Group of Imaginary Quadratic Fields}
\author[*]{Bart de Smit \\ 
\texttt{desmit@math.leidenuniv.nl}
}

\author[*]{\\Pavel Solomatin \\
\texttt{p.solomatin@math.leidenuniv.nl}
}

\affil[*]{Leiden University, Mathematical Department\\
Niels Bohrweg 1, 2333 CA Leiden}

\date{ Leiden, 2017}
\maketitle

\begin{abstract}
The main purpose of this paper is to extend results from \cite{Peter1} on isomorphism types of the abelianized absolute Galois group $\mathcal G_K^{ab}$, where $K$ denotes imaginary quadratic field. In particular, we will show that if the class number $h_K$ of an imaginary quadratic field $K$ different from $\mathbb Q(i)$, $\mathbb Q(\sqrt{-2})$ is a fixed prime number $p$ then there are only two isomorphism types  of $\mathcal G_K^{ab}$ which could occur. For instance, this result implies that imaginary quadratic fields of the discriminant $D_K$ belonging to the set $\{-35, -51, -91, -115, -123, -187, -235,$ $ -267,-403, -427 \}$ all have isomorphic abelian parts of their absolute Galois groups.

\end{abstract}

\textbf{\\ \\ \\ \\ \\ Acknowledgements:} 
The paper is a part of the PhD research of the second author under scientific direction of the first author. The second author was supported by the ALGANT scholarship during the research. Both authors would like to thank professors Hendrick Lenstra and Peter Stevenhagen for helpful discussions during the project.

\newpage
\section{Introduction}
\subsection{Motivation}
Let $K$ be a global field, i.e. either a finite extension of the field $\mathbb Q$ of rational numbers or a field of functions on a smooth projective geometrically connected curve $X$ over a finite field $k$. In the first case $K$ is a number field and in the later case $K$ is a global function field. An interesting question to ask is: what kind of information about $K$ one could recover from various subgroups of the absolute Galois group $\mathcal G_K = \Gal(K^{sep} : K)$ associated to $K$ ? The famous theorem on Neukirch and Uchida states that the isomorphism types of $\mathcal G_K$ considered as topological group determines the isomorphism class\footnote{up to Frobenius twist in the case of function fields.} of the field $K$. A question concerning the abelian part $\mathcal G_K^{ab}$ of $\mathcal G_K$ has attracted much attention since the work \cite{Onabe1} where in particular it was shown that there exists an example of imaginary quadratic fields with different class-groups and with isomorphic $\mathcal G_K^{ab}$. A dramatic improvement was achieved in the paper \cite{Peter1}, where authors produced a lot of new examples of non-isomorphic imaginary quadratic fields which share the same isomorphism type of $\mathcal G_K^{ab}$ and also showed that there are infinitely many isomorphism types of pro-finite groups which could occur as $\mathcal G_K^{ab}$. Moreover, based on their computations they made a conjecture that there are infinitely many imaginary quadratic fields with $\mathcal G_K^{ab} \simeq \widehat{\mathbb Z}^2 \times \prod_{n\in \mathbb N} \mathbb Z/n\mathbb Z $, where $\widehat{\mathbb Z}$ denotes the group of pro-finite integers.

Motivated by the above results authors of the present paper started working on the question about isomorphism type of $\mathcal G_K^{ab}$ where $K$ denotes a global function field.  For a global function field $K$ of characteristic $p$ with the exact constant field $\mathbb F_{q}$, $q=p^n$ we defined the invariant $d_K$ as a natural number such that  $n = p^{k} d_{K}$ with $\gcd(d_{K}, p)=1$. Let $\CL^{0}(K)$ denotes the degree zero part of the class-group of $K$. In other words $\CL^{0}(K)$ is the abelian group of $\mathbb F_q$-rational points of the Jacobian variety associated to the curve $X$.  In the pre-print \cite{FFAb} we proved the following result:
\begin{theorem}\label{main}
Suppose $K$ and $K'$ are two global function fields, then $\mathcal G^{ab}_K \simeq \mathcal G^{ab}_{K'}$ as pro-finite groups if and only if the following three conditions hold:
\begin{enumerate}
	\item $K$ and $K'$ share the same characteristic $p$;
	\item Invariants $d_K$ and $d_{K'}$ coincide: $d_K = d_{K'}$; 
	\item The non $p$-parts of class-groups of $K$ and $K'$ are isomorphic: $$\CL^{0}_{non-p} (K) \simeq \CL^{0}_{non-p} (K').$$
\end{enumerate}
In particular, two function fields with the same exact constant filed $\mathbb F_q$ have isomorphic $\mathcal G^{ab}_{K}$ if and only if they have isomorphic $\CL^{0}_{non-p} (K)$.
\end{theorem}

An important remark is our proof actually provides \emph{a description of isomorphism type of $\mathcal G_{K}^{ab}$}. Let $\mathcal T_K$ denotes the topological closure of the torsion of $\mathcal G^{ab}_K$. We showed that the isomorphism type of $\mathcal T_K$ depends only on the cardinality $q$ of the constant field of $K$ and actually gave an explicit description of the group $\mathcal T_K$ in terms of $q$ only. We also constructed an isomorphism of abelian groups: $$\CL^{0}_{non-p} (K) \simeq (\mathcal G^{ab}_K/\mathcal T_K)[\tor].$$ In the proof of the theorem \ref{main} we showed that the isomorphism type of $\mathcal G_{K}^{ab}$ is determined by isomorphism types of these two groups: $\mathcal T_q$ and $\CL^{0}_{non-p} (K) $. The main step is to prove that given groups $\mathcal T_K $ and $\CL^{0}_{non-p} (K)$ there exists a unique isomorphism type of a pro-finite abelian group $\mathcal D_K$ such that  the following holds:
\begin{enumerate}
	\item There exists an exact sequence: $0 \to \mathcal T_K \to \mathcal D_K \to \CL^{0}_{non-p} (K) \to 0 $;
	\item All torsion elements of $\mathcal D_K$ are in $\mathcal T_K$. 
\end{enumerate}

Finally, combining all results of our paper one has:

\begin{cor}
For a global function field $K$ of characteristic $p$ there exists an isomorphism of topological groups: $$\mathcal G_K^{ab} \simeq \widehat{\mathbb Z} \times \mathbb Z_p^{\infty} \times \mathcal D_K .$$
\end{cor}

Our result in the imaginary quadratic field case is quite similar to this statement. 

\subsection{The Statement of the Theorem}
The main purpose of the present paper is to use technique from \cite{FFAb} in order to extend results of the paper \cite{Peter1}. Let $K$ be an imaginary quadratic field different from $\mathbb Q(i)$, $\mathbb Q(\sqrt{-2})$. Let $\mathcal T = \prod_n \mathbb Z/n\mathbb Z$ and let $\CL(K)$ denotes the ideal class group of $K$. Summarising results of \cite{Peter1} we have: 
\begin{theorem}\label{Pet1}
In the above settings the following holds:
	\begin{enumerate}
		\item There exists an exact sequence of topological groups: $0\to \widehat{\mathbb Z}^2 \times \mathcal T \to \mathcal G_K^{ab} \to~\CL(K)~\to~0$;
		\item The topological closure $ \overbar{ \mathcal G_{K}^{ab}[\tor]}$ of the torsion subgroup of $\mathcal G_K^{ab}$ is $\mathcal T$;
		\item The torsion subgroup of the quotient $  \mathcal G_{K}^{ab} / \mathcal T$ is trivial if and only if  $\mathcal G_{K}^{ab} \simeq  \widehat{\mathbb Z}^2 \times \mathcal T$;
		\item There exist an injective map from $(\mathcal G_{K}^{ab} / \mathcal T)[\tor]$ to $\CL(K)$ and an algorithm with input $K$ and output whether the group $(\mathcal G_{K}^{ab} / \mathcal T)[\tor]$ is trivial or not.  
	\end{enumerate}
\end{theorem}
\begin{proof}
See theorem 3.5, 4.4 and 5.1 from \cite{Peter1}.
\end{proof}

Let us call the image of $(\mathcal G_{K}^{ab} / \mathcal T)[\tor]$ in $\CL(K)$ as $\CL^{split}(K)$. Roughly speaking our main result states that isomorphism type of $\mathcal G_K^{ab}$ is uniquely determined by the isomorphism type of $\CL^{split}(K)$. More concretely, we will prove that given groups $\mathcal T$ and $\CL^{split}(K)$ there exists a unique isomorphism type of a pro-finite abelian group $\mathcal D_K$ such that  the following holds:
\begin{enumerate}
	\item There exists an exact sequence: $0 \to \mathcal T \to \mathcal D_K \to \CL^{split}(K) \to 0 $;
	\item All torsion elements of $\mathcal D_K$ are in $\mathcal T$. 
\end{enumerate}

Then the main result of the present paper could be stated as:

\begin{theorem}\label{MainG1}
Let $K$ be an imaginary quadratic field different from $\mathbb Q(i)$, $\mathbb Q(\sqrt{-2})$. There exists an isomorphism of topological groups $\mathcal G_K^{ab} \simeq \mathcal D_K \times \widehat{\mathbb Z}^2$. 
\end{theorem}

The above theorem extends results of the theorem \ref{Pet1} as follows: 

\begin{cor}
For a fixed prime number $p$ and an imaginary quadratic field $K$ with class-number $h_K = p$ there are only two isomorphism types of $\mathcal G_K^{ab}$ which could occur: either $\CL^{split}(K) =0$ or $\CL^{split}(K) \simeq \mathbb Z/p\mathbb Z$. In particular, it was shown in \cite{Peter1} that imaginary quadratic fields with the discriminant $D_K$ occurring in the list $\{-35, -51, -91, -115,$  $-123, -187,-235, -267,-403, -427 \}$ all have class-number 2 and have non-trivial $\CL^{split}(K)$, therefore they all share the same isomorphism class of $\mathcal G_K^{ab}$.  
\end{cor}

\section{The Proof}
Our goal in this section is to prove theorem \ref{MainG1}. We will do this in two steps. First, we will show that there exist a pro-finite group $\mathcal D_K$ and an isomorphism $\mathcal G_K^{ab} \simeq \mathcal D_K \times \widehat{\mathbb Z}^2$. Then we will show that the group $\mathcal D_K$ is uniquely determined by the isomorphism class of the abelian group $\CL^{split}(K)$, provided $K \ne \mathbb Q(i)$, $\mathbb Q(\sqrt{-2})$. 

Since each pro-finite abelian group is isomorphic to the limit of finite abelian groups, by the Chinese remainder theorem we have that it is also isomorphic to the product over prime numbers of its primary components. We will work with these components separately instead of working with the whole group. 

\subsection{Proof of Splitting}
Consider the exact sequence mentioned in the theorem \ref{Pet1}: 

\begin{equation} \label{esn} 
0\to \widehat{\mathbb Z}^2 \times \mathcal T \to \mathcal G_K^{ab} \to \CL(K) \to 0
\end{equation}

 Taking a prime number $l$ we get the following exact sequence of pro-$l$ abelian groups:

\begin{equation} \label{esnl} 
0\to {\mathbb Z_l}^2 \times \mathcal T_l \to \mathcal G_{K,l}^{ab} \to \CL_{l}(K) \to 0,
\end{equation}
where $\mathcal T_l = \prod_{k \in \mathbb N} \mathbb Z/l^{k} \mathbb Z$ and $\mathbb Z_l$ denotes the group of $l$-adic integers.
If $\CL_{l}(K)$ is the trivial group then obviously $\mathcal G_{K,l}^{ab} \simeq {\mathbb Z_l}^2 \times \mathcal T_l$. Our goal is to describe the isomorphism type of $\mathcal G_{K,l}^{ab} $ in the case where $\CL_{l}(K)$ is not the trivial group. By the theorem \ref{Pet1} we know that $\mathcal T_l $ is the closure of the torsion subgroup of $\mathcal G_{K,l}^{ab}$. Note that $\mathcal T_l$ is a closed subgroup and hence the quotient is also pro-$l$ group. Taking the quotient of the sequence \ref{esnl} by $\mathcal T_l$ we obtain:
$$0\to {\mathbb Z_l}^2 \to \mathcal G_{K,l}^{ab}/\mathcal T_l \to \CL_{l}(K) \to 0.$$ 
 Since $\mathbb Z_l$ is torsion free we have $(\mathcal G_{K,l}^{ab}/\mathcal T_l)[\tor]$ maps invectively to $\CL_{l}(K)$ which is finite. Denoting the group $\mathcal G_{K,l}^{ab}/\mathcal T_l$ by $B_l$ we get isomorphism of topological groups\footnote{This is true because $B_l[\tor]$ is finite.}: $B_l \simeq B_l[\tor]\oplus B'_l $, where $B'_l$ denotes the non-torsion part of $B_l$. Since $\mathbb Z_l$ is torsion free we also have the following exact sequence: 
$$0\to {\mathbb Z_l}^2 \to B_l' \to \CL_{l}(K)/\phi(B_l[\tor]) \to 0 .$$
 Since $B'_l$ is torsion free this exact sequence implies $B'_l$ is a free $\mathbb Z_l$-module of rank two and hence $B'_l \simeq \mathbb Z_l^2$. 

Let us denote the quotient map $\mathcal G_{K,l}^{ab} \to \CL_{l}(K)$ by $\phi$. In notations from the introduction $\phi(B_l[\tor])  = \CL^{split}(K)$. Consider the pre-image $\mathcal D_l    \subset \mathcal G_{K,l}^{ab}$ of the group $\phi(B_l[\tor]) \subset \CL_{l}(K)$. Note that $\mathcal D_l$ is closed subgroup and we have the following exact sequence: 

$$0\to \mathcal T_l \to \mathcal D_l \to \phi(B_l[\tor]) \to 0 .$$
 
 Summing up we have the following commutative diagram of pro-$l$ abelian groups: 
\[\xymatrix{
& 0  &  0        & 0  &\\
0 \ar[r] & \mathbb Z_l^{2} \ar[u] \ar[r]         & B'_l \ar[u]\ar[r]  			&\CL_{l}(K) /\phi(B_l[\tor])  \ar[u]\ar[r]  & 0\\
0 \ar[r] & \mathcal T_{l} \times \mathbb Z_l^{2}\ar[u] \ar[r]          &\mathcal G_{K,l}^{ab} \ar[u] \ar[r]^{\phi}    			        &\CL_{l}(K) \ar[u] \ar[r] &0 \\ 
0 \ar[r] & \mathcal T_{l}\ar[r]\ar[u]                   & \mathcal D_l \ar[r] \ar[u]			        & \phi(B_l[\tor])\ar[r] \ar[u]  &0\\
& 0 \ar[u] &  0  \ar[u]        & 0 \ar[u] &\\
}\]

Now consider the exact sequence coming from the medium column of the above diagram: 

$$ 0\to \mathcal D_l \to \mathcal G_{K,l}^{ab} \to B'_l \to 0 .$$

We know that $B'_l \simeq \mathbb Z_l^{2}$, but $\mathbb Z_l$ is a projective module and hence we could split this sequence to obtain isomorphism: $ \mathcal G_{K,l}^{ab} \simeq \mathcal D_l \times B'_l \simeq \mathcal D_l \times \mathbb Z_l^2$, which finishes the first step. 

\subsection{Proof of Uniqueness} 
Our main result is to show that the group $\mathcal D_l$ is determined uniquely by the isomorphism type of $\phi(B_l[\tor]) = \CL^{split}(K)$. Consider the exact sequence: 
$$ 0 \to \mathcal T_l \to \mathcal D_l \to  \CL^{split}(K) \to 0$$
We know that the closure of the torsion subgroup of $\mathcal G_{K,l}^{ab}$ is $\mathcal T_l$ therefore, $\mathcal D_l$ contains no torsion elements apart from elements of $\mathcal T_l$. 
Our goal is to prove:

\begin{theorem}\label{progroup}
Given pro-$l$ abelian group $\mathcal T_l \simeq \prod_{k\in \mathbb N} \mathbb Z/l^{k}\mathbb Z$ and a finite abelian $l$-group $A$ there exists a unique isomorphism type of pro-$l$ abelian group $\mathcal D_l$ such that the following holds:
\begin{enumerate}
	\item There exists an exact sequence of pro-$l$ abelian groups: $ 0 \to \mathcal T_l \to \mathcal D_l \to  A \to 0$;
	\item The topological closure of the torsion subgroup of $\mathcal D_l$ is $\mathcal T_l$: $ \overbar{ (\mathcal D_l)[\tor]} = \mathcal T_l $. 
\end{enumerate}
\end{theorem}

The key idea in the proof is to use the Pontryagin duality for locally compact abelian groups to reduce the question about pro-$l$ groups to the more elementary question about discrete torsion groups and then use the following theorem:

\begin{theorem}\label{group}
Let $\{ C_i \}$ be a countable set of finite cyclic abelian $l$-groups with orders of $C_i$ are not bounded as $i$ tends to infinity and let $A$ be any finite abelian $l$-group.
Then up to isomorphism there exists a unique torsion abelian $l$-group $B$ satisfying two following conditions:
\begin{enumerate}
	\item There exists an exact sequence: $0 \to A \to B \to \oplus_{i \ge 1} C_i \to 0$;
	\item $A$ is the union of all divisible elements of $B$: $A = \cap_{n \ge 1} nB$.
\end{enumerate} 
\end{theorem}
\begin{proof}
See \cite{FFAb}.
\end{proof}

We will show that the Pontryagin dual of the exact sequence $ 0 \to \mathcal T_l \to \mathcal D_l \to  \CL^{split}(K) \to 0$ satisfies conditions of the theorem \ref{group} and therefore $\mathcal D_{l}$ is uniquely determined, since its Pontryagin dual $(D_l)^{\vee}$ is uniquely determined.

\subsubsection{The Pontryagin duality}
We need to recall some properties of the Pontryagin duality for locally compact abelian groups. A good reference including some historical discussion is \cite{Pont}. Let $\T$ be the topological group $\mathbb R / \mathbb Z$ given with the quotient topology. If $A$ is any locally compact abelian group one consider the Pontryagin dual $A^{\vee}$ of $A$ which is the group of all continuous homomorphism from $A$ to $\T$ : $$ A^{\vee} = \Hom(A, \T). $$ 
This group has the so-called compact-open topology and is a topological group. Here we list some properties of the Pontryagin duality we use during the proof: 
 \begin{enumerate}
 	\item The Pontryagin duality is a contra-variant functor from the category of locally compact abelian groups to itself;
	\item If $A$ is a finite abelian group treated with the discrete topology then $A^{\vee} \simeq A$ non-canonically;
	\item We have the canonical isomorphism: $(A^{\vee})^{\vee} \simeq A$;
	\item The Pontryagin dual to the pro-finite abelian group $A$ is a discrete discrete torsion group and vice versa;
	\item The Pontryagin duality sends direct products to direct sums and vice versa;
\end{enumerate}
Having stated these facts we are able to finish the proof.
\subsubsection{The final step}

In a settings of the theorem \ref{progroup} the multiplication by $l^n$ map induces the following commutative diagram:

\[\xymatrix{
0 \ar[r] & \mathcal T_{l}[l^{n}] \ar@{^(->}[d] \ar@{^(->>}[r]         & \mathcal D_{l}[l^{n}]\ar@{^(->}[d]\ar[r]^{0}  			&A[l^n]\ar@{^(->}[d]  &\\
0 \ar[r] & \mathcal T_{l}\ar[d]^{l^n} \ar[r]          &\mathcal D_{l}\ar[d]^{l^n} \ar[r]    			        &A\ar[d]^{l^n} \ar[r] &0 \\ 
0 \ar[r] & \mathcal T_{l}\ar[r]\ar@{->>}[d]                   &\mathcal D_{l} \ar[r] \ar@{->>}[d]			        & A \ar[r] \ar@{->>}[d]  &0\\
 & \mathcal T_{l}/l^n\mathcal T_{l}\ar[r] & \mathcal D_{l} /l^n \mathcal D_{l}  \ar[r]        &A/l^n A \ar[r] &0\\
}\]

Since any torsion element $x$ of $\mathcal D_l$ is in $\mathcal T_l$ the map from $\mathcal D_l[l^{n}]$ to $A[l^n]$ is the zero map and the map from $\mathcal T_{l}[l^{n}]$ to $ \mathcal D_l[l^{n}]$ is an isomorphism. Now applying the Pontryagin duality to the above diagram we get: 
\[\xymatrix{
0  & (\mathcal T_{l}[l^{n}] )^{\vee} \ar[l]          & (\mathcal D_{l}[l^{n}] )^{\vee}\ar@{^(->>}[l]   			&(A[l^n])^{\vee} \ar[l]^{0} &  \\
0  & (\mathcal T_{l})^{\vee}\ar@{->>}[u] \ar[l]          &(\mathcal D_{l})^{\vee}\ar@{->>}[u] \ar[l]    			        &(A)^{\vee}\ar@{->>}[u] \ar[l] &0\ar[l] \\ 
0  & (\mathcal T_{l})^{\vee}\ar[l]\ar[u]^{l^n}                   & (\mathcal D_{l})^{\vee} \ar[u]^{l^n} \ar[l]			        & (A)^{\vee}\ar[l] \ar[u]^{l^n}  &0\ar[l]\\
  & (\mathcal T_{l}/l^n\mathcal T_{l})^{\vee} \ar@{^(->}[u] & (\mathcal D_{l}/l^n \mathcal D_{l})^{\vee}\ar@{^(->}[u] \ar[l]        &(A/l^n A)^{\vee}\ar@{^(->}[u]\ar[l] &0\ar[l]\\
}\]

Note that $(\mathcal T_{l})^{\vee}$ is isomorphic to the direct sum of cyclic groups $(\mathcal T_{l})^{\vee} \simeq \oplus_{k \in \mathbb N} \mathbb Z/l^{k}\mathbb Z$ and therefore $\cap_n l^{n} (\mathcal T_{l})^{\vee} = \{0 \}$. It means we have $(\cap_n l^{n} (\mathcal D_l)^{\vee}) \subset (A)^{\vee}$. Our goal is to show that $(\cap_n l^{n} (\mathcal D_l)^{\vee}) = (A)^{\vee}$.
\begin{lem}
Given any non-zero element $x$ of $ (A)^{\vee} \subset (\mathcal D_l)^{\vee}$ and any natural number $n$ there exists an element $c_x \in (\mathcal D_l)^{\vee}$ such that $l^n c_x = x$. 
\end{lem}
\begin{proof}
For fixed $n$ consider the above diagram. Since the second row is exact the image of $x$ in $(\mathcal T_{l})^{\vee}$ is zero. Then its image in $\mathcal (\mathcal T_{l}[l^{n}] )^{\vee}$ is also zero. Since $(\mathcal T_{l}[l^{n}] )^{\vee}  \simeq (\mathcal D_l[l^{n}] )^{\vee}$ it means that image of the non-zero element $x$ in $(\mathcal D_l[l^{n}] )^{\vee}$ is zero. Since the second column is exact this means that $x$ lies in the image of the multiplication by $l^n$ map from $(\mathcal D_l)^{\vee} $ to $(\mathcal D_l)^{\vee} $ and therefore there exists $c_x$ such that $l^{n} c_x = x$.
\end{proof}

It means that we have proved:
\begin{cor}\label{cor1}
The exact sequence $0\leftarrow (\mathcal T_{l})^{\vee} \leftarrow (\mathcal D_l)^{\vee} \leftarrow (A)^{\vee} \leftarrow 0$ satisfies conditions of the theorem \ref{group}.
\end{cor}

and therefore $\mathcal D_l$ is uniquely determined since its Pontryagin dual $(\mathcal D_l)^{\vee}$ is uniquely determined by the theorem \ref{group}.
 
\newpage

\bibliography{mybib}{}
\bibliographystyle{plain}

\newpage

\tableofcontents

\end{document}